\documentclass{article}

\usepackage[english]{babel}
\usepackage[utf8x]{inputenc}
\usepackage[T1]{fontenc}

\usepackage[letterpaper,top=3cm,bottom=2cm,left=3cm,right=3cm,marginparwidth=1.75cm]{geometry}

\usepackage{bm}
\usepackage{bbm}
\usepackage{enumerate}
\usepackage{graphicx}
\usepackage[colorinlistoftodos]{todonotes}
\usepackage{booktabs}
\usepackage{algpseudocode}
\usepackage{multirow}
\usepackage{amssymb,amsmath}
\numberwithin{equation}{section}
\usepackage{bm}
\usepackage{amsthm}
\newtheorem{lemma}{Lemma}
\newtheorem{theorem}{Theorem}

\newcommand\blfootnote[1]{%
  \begingroup
  \renewcommand\thefootnote{}\footnote{#1}%
  \addtocounter{footnote}{-1}%
  \endgroup
}
\newcommand{\cksmalla}{C_1}
\newcommand{\cLlaa}{C_2}
\newcommand{\cLlab}{C_3}
\newcommand{\cCorEthree}{C_4}
\newcommand{\cCorEthreeb}{\cLlab\nu}
\newcommand{\cLnalowera}{C_5}
\newcommand{\cLnalowerb}{C_6}
\newcommand{\cBlocka}{C_7}
\newcommand{\cBlockb}{C_8}
\newcommand{\cfinala}{C_9}
\newcommand{\cfinalb}{C_{10}}

\newcommand{\Ethreenl}{G^{(n)}_\ell(\epsilon)}
\newcommand{\Ethreen}{G^{(n)}(\epsilon)}
\newcommand{\Efournk}{H^{(n)}_k}
\newcommand{\Efourn}{H^{(n)}}
\newcommand{\Efiven}{J^{(n)}}
\newcommand{\Esixnk}{I^{(n)}_k}
\newcommand{\Esixn}{I^{(n)}}
\newcommand{\tildel}[1]{\hat{#1}}

\DeclareMathOperator{\E}{\mathbb{E}}
\DeclareMathOperator{\bbP}{\mathbb{P}}

\DeclareMathOperator{\Var}{Var}

\title{On the Variance of the Length of the Longest Common Subsequences in Random Words With an Omitted Letter}
\author{ Christian Houdr\'e\thanks{School of Mathematics,
Georgia Institute of Technology,
686 Cherry Street,
Atlanta, GA 30332-0160 USA, houdre@math.gatech.edu. Research supported in part by the grant $\# 246283$ and $\# 524678$ from the Simons Foundation.}
\and
Qingqing Liu\thanks{School of Mathematics,
Georgia Institute of Technology,
686 Cherry Street,
Atlanta, GA 30332-0160 USA, qqliu@gatech.edu}}

\begin{document}
\maketitle


\begin{abstract}
We investigate the variance of the length of the longest common subsequences of two independent random words of size $n$, where the letters of one word are i.i.d.\ uniformly drawn from  $\{\alpha_1, \alpha_2, \cdots, \alpha_m\}$, while the letters of the other word are i.i.d.\  drawn from $\{\alpha_1, \alpha_2, \cdots, \alpha_m, \alpha_{m+1}\}$, with probability $p > 0$ to be $\alpha_{m+1}$, and $(1-p)/m > 0$ for all the other letters. The order of the variance of this length is shown to be linear in $n$. 
\end{abstract}

\section{Introduction and Statement of Results}
Let $\bm{X} =(X_i)_{i \ge 1}$ and $\bm{Y} = (Y_i)_{i \ge 1}$ be two independent sequences of i.i.d.~random variables taking their values in a finite common alphabet $\mathcal{A}$, with $\bbP(X_1=\alpha) = p_{x,\alpha}\ge 0$ and $\bbP(Y_1=\alpha)=p_{y,\alpha}\ge 0$, $\alpha\in \mathcal{A}$. Let $LC_n$ be the largest $k$ such that there exist $1\le i_1 < \cdots <i_k \le n$ and $1\le j_1 < \cdots <j_k \le n$ with $X_{i_s}=Y_{j_s}$  for $s = 1,\ldots,k$, i.e., $LC_n$ denotes the length of the longest common subsequences of the random words $\bm{X}^{(n)} := X_1\cdots X_n$ and $\bm{Y}^{(n)} := Y_1\cdots Y_n$. 
The limiting behavior of the expectation of $LC_n$ has been extensively studied. In particular, if for all $\alpha\in\mathcal{A}$, $p_{x,\alpha}=p_{y,\alpha}=1/(\# \mathcal{A})$, where $\# \mathcal{A}$ denotes the cardinality of $\mathcal{A}$, the earliest result is due to Chv{\'a}tal and Sankoff \cite{chvatal_longest_1975}, who proved the existence of
\begin{equation*}
\gamma_m^*=\lim_{n\to\infty}\frac{\E LC_n}{n},
\end{equation*}
where $m$ denotes the alphabet size, showing also that $0.727273\le \gamma_2^* \le
0.905118$. Much work has since been done to improve these
bounds (\cite{deken_limit_1979}, \cite{chvatal_upper-bound_1983}, \cite{deken_probabilistic_1983}, \cite{dancik_expected_1994}, $\ldots$), and to date the best known bounds seem to be $0.788071 \le \gamma_2^* \le 0.826280$, see \cite{lueker_improved_2009}. These results have also been extended to multiple sequences and alphabet of size larger than two, e.g., see \cite{kiwi_speculated_2009}, \cite{liu_2017} and the references therein. 

The study of the variance of $LC_n$ is less complete. In case $p_{x,k} = p_{y,k} = p_k $ for $k=1,\ldots,m$, the Efron-Stein inequality implies, as shown in \cite{steele_efron-stein_1986}, that
\begin{equation*}
\Var LC_n \le n\left(1-\sum_{k=1}^m p_k^2\right).
\end{equation*}
. 
 
For lower bounds, linear order results are also proved in various biased instances (\cite{lember_standard_2009}, \cite{houdre_order_2016}, \cite{houdre_variance_2016}, \cite{lember_lower_2016}, \cite{gong_lower_2016}, \cite{amsalu_sparse_2016}, \cite{bonetto_fluctuations_2006},$\ldots$). For example, \cite{lember_standard_2009} and \cite{houdre_order_2016} assume that one of the letters has a significantly higher probability of appearing than any of the other letters in the alphabet, while \cite{bonetto_fluctuations_2006} assumes that one of the two sequences is binary while the other is a trinary one. Our paper extends the result of \cite{bonetto_fluctuations_2006} by removing the binary/trinary assumptions and provides precise estimates allowing us to go beyond the uniform case and to also deal with central moments.


To formally state our problem, let $\mathcal{A} := \mathcal{A}_{m+1} = \{\alpha_1, \alpha_2, \cdots, \alpha_m, \alpha_{m+1} \}$, and let the letters distribution of $\bm{X}$ to be such that
$$  \mathbb{P}(X_1 = \alpha_1 ) = \cdots = \mathbb{P}(X_1 = \alpha_m) = \frac{1-p}{m} > 0, \hspace{.1 in} \mathbb{P}(X_1 = \alpha_{m+1} )  = p > 0 ,$$
while the letters distribution of $\bm{Y}$ is such that
$$\mathbb{P}(Y_1 = \alpha_1 ) = \cdots = \mathbb{P}(Y_1 = \alpha_m )  = \frac{1}{m}.$$


To start with, an upper bound on the variance of $LC_n$ is shown to be 
$$\Var LC_n \le \frac{n}{2}\left( 2 - p^2 - \frac{1 + (1-p)^2}{m} \right),$$
for all $n \in \mathbb{N}$. 
Indeed, the Efron–Stein inequality 
states that:
\begin{equation}
\label{eqn:efronstein}
\Var S \le \frac{1}{2}  \sum_{i =1}^{n} \mathbb{E}(S-S_i)^2, 
\end{equation}
where, $S = S(Z_1, Z_2, \cdots, Z_n)$ and $S_i = S(Z_1, Z_2, \cdots, Z_{i - 1}, \hat{Z_i}, Z_{i+1}, \cdots  Z_n)$, and where $(Z_i)_{1\le i \le n}$ and  $(\hat{Z_i})_{1\le i \le n}$ are independent copies of each other. 

Now following \cite{steele_efron-stein_1986},
\begin{align*}
 & \mathbb{E}| LC_n - LC_n(X_1\cdots X_{i-1} \hat{X_i} X_{i+1} \cdots X_n; Y_1 \cdots Y_n ) |^2  \\
& \hspace{3em} =  \mathbb{E} \left ( |LC_n - LC_n(X_1\cdots  X_{i-1} \hat{X_i}  X_{i+1} \cdots X_n; Y_1 \cdots Y_n)|^2 \bm{1}_{X_i \neq \hat{X_i}} \right) \\
& \hspace{3em} \le  \mathbb{P}(X_i \neq \hat{X_i})= 1 - \sum_{i = 1}^{m+1} \left( \mathbb{P}(X_1 = \alpha_i) \right)^2 \\
& \hspace{3em} = 1 - m \left( \frac{1-p}{m} \right)^2 - p^2\\
& \hspace{3em} = (1-p) \left( 1 - \frac{1}{m} + p \left( 1 + \frac{1}{m}\right) \right),
\end{align*}

since when replacing $X_i$ by $\hat{X_i}$, $LC_n$ changes by at most $1$ and at least $-1$.
Similarly, 
\begin{align*}
\mathbb{E}| LC_n - LC_n(X_1\cdots X_n; Y_1 \cdots  Y_{i-1} \hat{Y_i}  Y_{i+1}\cdots Y_n ) |^2  
& \le 1 - \sum_{i = 1}^{m} \left( \mathbb{P}(Y_1 = \alpha_i) \right)^2 \\
&= 1 - \frac{1}{m}.
\end{align*}

Applying \eqref{eqn:efronstein} and combining the two bounds above give, 
\begin{align}
\label{eqn:varupp}
\Var LC_n & \le \frac{1}{2} \left\{ \left( (1-p) \left( 1 - \frac{1}{m} + p \left( 1 + \frac{1}{m}\right) \right) \right) n  + \left(1- \left( \frac{1}{m} \right) \right) n  \right\} \nonumber\\
& = \frac{n}{2}\left( 2 - p^2 - \frac{1 + (1-p)^2}{m} \right) .
\end{align}

To match the easy bound \eqref{eqn:varupp}, we can now state the main result of this paper.  

\begin{theorem}
\label{theorem:main}
There exists a constant $C = C(p, m) > 0$ independent of $n$, such that for all $n \ge 1$,
\begin{equation}
   \Var LC_n \ge Cn. 
   \label{eq:main}
\end{equation}
\end{theorem}

This theorem, combined with the upper bound \eqref{eqn:varupp}, gives a linear order, in $n$, for the variance of $LC_n$, and we refer the reader to Section~\ref{sec:constants} for an estimate on $C$.

\section{Proof of Theorem~\ref{theorem:main}}
\label{sec:mainpf}
The scheme of the proof elaborates and extends elements of of \cite{bonetto_fluctuations_2006} and \cite{houdre_order_2016}. So, let $N$ denote the number of letters $\alpha_{m+1}$ in the random word $\bm{X}^{(n)}$. Clearly, $N$ is a binomial random variable with parameter $n$ and $p$. Moreover, let $\tilde{\bm{X}}^{(n)}:=X_{i_1}\cdots X_{i_k}$, where $1\le i_1 < \cdots <  i_k \le n$, $X_j\neq\alpha_{m+1}$ for all $j\in\{i_1,\ldots,i_k\}$ and $X_j=\alpha_{m+1}$ for all $j \in \{1, 2, \ldots, n\}\backslash \{i_1,\ldots,i_k\}$. 
In words, $\tilde{\bm{X}}^{(n)}$ is the subword of $\bm{X}^{(n)}$ made only of non-$\alpha_{m+1}$ letters. 
To prove our main theorem, we will recursively define a finite random sequence $\bm{Z}^{(1)}, \bm{Z}^{(2)}, \ldots, \bm{Z}^{(n)} $, where each $\bm{Z}^{(k)}$ has length $k$, by inserting uniformly at random and at a uniform random location a letter from $\{\alpha_1, \alpha_2, \ldots, \alpha_m\}$ to the previous $\bm{Z}^{(k-1)}$. 

To formally describe the defining mechanism, let $\{U_k\}_{1\le k \le n }$ and $\{T_k\}_{3 \le k \le n}$ be two independent sequences of random variables, where $\{U_k\}_{1 \le k \le n}$ is a sequence of i.i.d.\ uniform random variables on $\{\alpha_1, \alpha_2, \ldots, \alpha_m\}$, and $\{T_k\}_{3 \le k\le n}$ is a sequence of independent random variables uniform on $\{2, 3, \ldots, k-1\}$, $k \ge 3$. 

Then as in \cite{bonetto_fluctuations_2006}, recursively define the sequence $\bm{Z}^{(k)}$ via:
\begin{enumerate}[(1)]
\item $\bm{Z}^{(1)} = U_1$.
\item $\bm{Z}^{(2)} = U_1U_2$. 
\item For $k \ge 2$, given $\bm{Z}^{(k)} = Z_1^kZ_2^k\cdots Z_k^k$, let $\bm{Z}^{(k+1)}$ be as follows:
\begin{itemize}
\item For all $j < T_{k+1}$, let $$Z_j^{k+1} = Z_j^k.$$
\item For $j = T_{k+1}$, let $$Z_j^{k+1} = U_{k+1}.$$
\item For all $j$ such that $T_{k+1} < j \le k+1$, let $$Z_j^{k+1} = Z_{j-1}^k.$$
\end{itemize}
\end{enumerate}

Hence,  $\{Z_i^k \}_{1\le i\le k \le n}$ is a triangular array of uniform random variables with values in $\{\alpha_1, \alpha_2, \ldots, \alpha_m\}$, 
and finding the relation between $\bm{Z}^{(n-N)}$ and  $\tilde{\bm{X}}^{(n)}$ is the purpose of our next lemma whose proof is akin to a corresponding proof in \cite{houdre_order_2016}. 

\begin{lemma}
\label{lemma:LCn_dist}
For any $n\ge 1$ and $1\le k \le n$, 
$$\bm{Z}^{(k)} \stackrel{\text{d}}{=} (\tilde{\bm{X}}^{(n)} | N = n-k),$$
and moreover, 
$$\bm{Z}^{(n-N)} \stackrel{\text{d}}{=} \tilde{\bm{X}}^{(n)},$$ 
where $\stackrel{\text{d}}{=}$ denotes equality in distribution.
\end{lemma}

\begin{proof}
The proof is by induction on $k$. Let $k = 1$, by definition, $\bm{Z}^{(1)} = U_1$, which has the same distribution as $(\tilde{\bm{X}}^{(n)}|N = n -1)$. Next, assume that 
$$\bm{Z}^{(k)} \stackrel{\text{d}}{=} (\tilde{\bm{X}}^{(n)} | N = n-k),\, 2\le k \le n-1,$$
and so for any $(\alpha_{j_1}, \alpha_{j_2}, \ldots, \alpha_{j_k} ) \in \mathcal{A}^k$,  
$$\mathbb{P}\left( (Z_1^k, Z_2^k, \ldots, Z_k^k) = (\alpha_{j_1}, \alpha_{j_2}, \ldots, \alpha_{j_k} ) \right) = \left( \frac{1}{m} \right)^k.$$

Then,
\begin{align*}
&\hspace{-2em}\mathbb{P}  \left( (Z_1^{k+1}, Z_2^{k+1}, \ldots, Z_{k+1}^{k+1}) = (\alpha_{j'_1}, \alpha_{j'_2}, \ldots, \alpha_{j'_{k+1}} ) \right)  \\
&=  \sum_{t=2}^{k} \mathbb{P}\left( (Z_1^{k+1}, Z_2^{k+1}, \ldots, Z_{k+1}^{k+1}) = (\alpha_{j'_1}, \alpha_{j'_2}, \ldots, \alpha_{j'_{k+1}} ) | T_{k+1} = t \right)  \mathbb{P}(T_{k+1} = t) \\
&=  \sum_{t=2}^{k} \mathbb{P}\left( (Z_1^k, \ldots, Z_{t-1}^{k}, Z_{t}^{k}, \ldots,  Z_{k}^{k}) = (\alpha_{j'_1}, \ldots,\alpha_{j'_{t-1}},\alpha_{j'_{t+1}},\ldots, \alpha_{j'_{k+1}} ) \right)  \mathbb{P}( U_{k+1} = \alpha_{j'_t}) \mathbb{P}(T_{k+1} = t) \\
&=   \sum_{t=2}^{k} \left( \frac{1}{m} \right)^{k} \frac{1}{m} \frac{1}{k-1} \\
&=  \left( \frac{1}{m} \right)^{k+1}.
\end{align*}
Thus, $$\bm{Z}^{(k+1)} \stackrel{\text{d}}{=} (\tilde{\bm{X}}^{(n)} | N = n-k-1).$$

To prove the second part of the lemma, from the independence of $N$ and $\bm{Z}^{(n-k)}$, for any $u \in \mathbb{R}^{n-k}$,
\begin{align*}
\mathbb{E} e^{i <u, \tilde{\bm{X}}^{(n)}>}
&=  \sum_{k = 0}^n  \mathbb{E} \left( e^{i<u, \tilde{\bm{X}}^{(n)} >} | N = k \right) \mathbb{P}(N = k)\\
&=  \sum_{k = 0}^n  \mathbb{E} \left( e^{i<u, \bm{Z}^{(n-k)}>}  \right) \mathbb{P}(N = k) \\ 
&=  \sum_{k = 0}^n  \mathbb{E} \left( e^{i<u, \bm{Z}^{(n-k)}>} |N=k  \right) \mathbb{P}(N = k) \\
&=  \sum_{k = 0}^n  \mathbb{E} \left( e^{i<u, \bm{Z}^{(n-N)}>} |N=k  \right) \mathbb{P}(N = k) \\
&=  \mathbb{E} e^{i<u, \bm{Z}^{(n-N)}>} .
\end{align*}
Thus, 
$$\bm{Z}^{(n-N)} \stackrel{\text{d}}{=} \tilde{\bm{X}}^{(n)}.$$
\end{proof}

Now let $LC_n$ be the length of the longest common subsequences of $\bm{X}^{(n)}$ and $\bm{Y}^{(n)}$, and let $L_n(k)$ be the length of the longest common subsequences/subwords of $\bm{Z}^{(k)}$ and $\bm{Y}^{(n)}$. It follows from Lemma~\ref{lemma:LCn_dist} that,
\begin{equation}
LC_n \stackrel{\text{d}}{=} L_n(n - N),
\label{eq:eqdist}
\end{equation}
and therefore,
\begin{equation}
\Var LC_n = \Var(L_n(n - N)).
\label{eq:eqMr}
\end{equation}

In order to prove the main result, we will also need the following result taken from \cite{houdre_order_2016}. 

\begin{lemma}
\label{lemma:var_lowerbd}
Let $f: D \subset \mathbb{R}\to \mathbb{Z}$ satisfy a local reversed Lipschitz condition, i.e., let $h \ge 0$ and let $f$ be such that for any $i, j \in D$ with $j \ge i + h$, 
$$f(j) - f(i) \ge c (j - i), $$
for some $c > 0$. Let $T$ be a $D$-valued random variable with $\mathbb{E}|f(T)|^2 < \infty$, then 
$$ \Var f(T) \ge \frac{c^2}{2} \left(\Var(T) - h^2\right). $$
\end{lemma}



Next, let
\begin{equation}
O_n := \bigcap_{\substack{i,j \in I\\j  \ge i + h(n)}}\left\{  L_n(j) - L_n(i) \ge K (j-i) \right\},
\label{eq:On}
\end{equation}
where $I = [np-\sqrt{np(1-p)}, np + \sqrt{np(1-p)} ]$,  $K > 0$ is a constant which does not depend on $n$ ($K\le 1/2m$ will do, see Lemma \ref{lemma:On_intersection}), and where $h(n)$ will also be made precise later. The event $O_n$ can be viewed as the event where the map $k \to L_n(k)$ locally satisfies a reversed Lipschitz condition.

In Section~\ref{sec:On}, we will prove
\begin{theorem}
\label{theorem:On}
For all $n\ge 1$,
\begin{equation}
\mathbb{P}(O_n) \ge 1 - A e^{-B n} - n e^{-2K^2h(n)},
\label{eq:thmOn}
\end{equation}
where, $K$ is given in Lemma~\ref{lemma:On_intersection}, $A = \max\{ \cCorEthree, \cLnalowera, \cBlocka\}$, and $B = \min \{\cCorEthreeb,  \cLnalowerb, \cBlockb \}$, and these constants are given in \eqref{eq:pgn}, Lemma~\ref{lemma:Lna_lower}, and Lemma~\ref{lemma:block} respectively. 
\end{theorem}

Now with the help of Theorem \ref{theorem:On} we can provide the proof of our main result stated in Theorem~\ref{theorem:main}. 

\begin{proof}[Proof of Theorem~\ref{theorem:main}]
By \eqref{eq:eqMr}, it is sufficient to prove the lower bound for $\Var(L_n(n-N))$.
First as in \cite{houdre_order_2016}, with its notation, 
\begin{align}
 \Var(U|V) &\le  2^2\left(
\mathbb{E}(\left.(U-\mathbb{E}U)^2 \,\middle|\, V\right.)/2+
\mathbb{E}(\left.(\mathbb{E}(U|V)-\mathbb{E}U)^2 \,\middle|\, V\right.)/2
\right)\notag\\
&\le  2^2\mathbb{E}(\left.(U-\mathbb{E}U)^2 \,\middle|\, V\right.), \label{eq:cond_ineq}
\end{align}
and so, for any $n\ge 1$,
\begin{align}
\Var(L_n(n-N)) &\ge \frac{1}{2^2}\mathbb{E}\left(
\Var(L_n(n-N) \,|\, (L_n(n-k))_{0\le k\le n})
\right)
\notag\\
&=\frac{1}{2^2}\int_{\Omega} 
\Var(L_n(n-N) \,|\, (L_n(n-k))_{0\le k\le n}(\omega))\mathbb{P}(d\omega)
\notag\\
&\ge \frac{1}{2^2}\int_{O_n} 
\Var(L_n(n-N) \,|\, (L_n(n-k))_{0\le k\le n}(\omega))\mathbb{P}(d\omega).
\label{E:lowerbound1}
\end{align}
Since $N$ is independent of $(L_n(n-k))_{0 \le k \le n}$, and
from \eqref{eq:cond_ineq}, for each $\omega \in \Omega$,
\begin{align}\label{E:lowerbound2}
&\Var(L_n(n-N)|(L_n(n-k))_{0\le k \le n}(\omega))\notag\\
& \ge \Var(L_n(n-N)|(\!L_n(n-k))_{0\le k \le n}(\omega),
\mathbf{1}_{N\in I}=1)\mathbb{P}(N \in I | (L_n(n-k))_{0\le k \le n}(\omega))\notag\\
& = \Var(L_n(n-N)|(L_n(n-k))_{0\le k \le n}(\omega),
\mathbf{1}_{N\in I}=1)\mathbb{P}(N \in I),
\end{align}
where again, 
\begin{equation*}\label{eq:interval}
I=\left[np -\sqrt{n(1-p)p },np +\sqrt{n(1-p)p }\right].
\end{equation*}
Again, for each $\omega \in O_n$, from Lemma \ref{lemma:var_lowerbd}, and since $N$ is independent of $(L_n(n-k))_{0 \le k \le n}$,
\begin{multline}\label{E:lowerbound5-new}
 \Var(L_n(n-N)|(L_n(n-k))_{0\le k \le n} (\omega), \mathbf{1}_{N\in I }=1)
 \ge   \frac{K^2}{8} \left(\Var(N| \mathbf{1}_{N \in I }=1)-h(n)^2\right).
\end{multline}
Now, \eqref{E:lowerbound1}, \eqref{E:lowerbound2} and \eqref{E:lowerbound5-new} give
\begin{equation}\label{E:lowerbound4-new}
\Var(L_n(n-N))
\ge \frac{K^2}{8}\left(\Var(N|
\mathbf{1}_{N \in I }=1)-h(n)^2\right)\mathbb{P}(N \in I) \mathbb{P}(O_n),
\end{equation}
and it remains to estimate each one of the three terms on the right hand side of
\eqref{E:lowerbound4-new}. 
By the Berry-Ess\'een inequality, for all $n\ge 1$,
\begin{equation}\label{LCS:BE}
\left|\mathbb{P}(N \in I)- \frac{1}{\sqrt{2 \pi}}
\int_{-1}^1 e^{-\frac{x^2}{2}} dx\right|\le \frac{1}{\sqrt{np(1-p)}}.
\end{equation}
Moreover,
\begin{align}\label{condition1}
\Var(N|\mathbf{1}_{N \in I }=1)\notag
&= \mathbb{E}(\left( N-np+np-\mathbb{E}(N|\mathbf{1}_{N \in I }=1)\right)^2 | \mathbf{1}_{N \in I }=1)\notag\\
&\ge \left(\mathbb{E}((N-np)^2 | \mathbf{1}_{N \in I }=1)^{1/2}-|np-\mathbb{E}(N|\mathbf{1}_{N \in I }=1)|\right)^2,
\end{align}
and
\begin{align}
&|\mathbb{E}(N|\mathbf{1}_{N \in I }=1)-np| \notag\\
&\hspace{.5in} =\sqrt{np(1-p)}\left|\mathbb{E}\left(\frac{N-np}
{\sqrt{np(1-p)}}\Big| \mathbf{1}_{N \in I }=1 \right)\right|\notag\\
&\hspace{.5in} = \sqrt{np(1-p)}\frac{\left|\! F_n(1)-\Phi(1)+F_n(-1)-\Phi(-1)-\int_{-1}^{1}(F_n(x)-\Phi(x))dx\right|}{\mathbb{P}(N \in I)}\notag\\
&\hspace{.5in} \le \sqrt{np(1-p)} \frac{4 \max_{x \in [-1,1]}|F_n(x)-\Phi(x)|}
{\mathbb{P}(N \in I)}\notag\\
& \hspace{.5in} \le  \frac{2}{\int_{-1}^1 e^{-\frac{x^2}{2}}dx/{\sqrt{2 \pi}} - 1/{\sqrt{np(1-p)}}},\label{condition2}
\end{align}
where $F_n$ is the distribution functions of ${(N-np)}/{\sqrt{np(1-p)}}$, while
$\Phi$ is the standard normal one. Likewise,
\begin{align}
&\mathbb{E}(|N-np|^2 | \mathbf{1}_{N \in I }=1) \notag\\
&\hspace{.5in} \ge (n p(1-p))\frac{\int_{-1}^1 |x|^2 d\Phi(x) -4\max_{x \in [-1,1]}|F_n(x)-\Phi(x)|}{\mathbb{P}(N \in I)}   \notag\\
&\hspace{.5in} \ge (np(1-p)) \frac{\int_{-1}^1 |x|^2 e^{-\frac{x^2}{2}}dx-2\sqrt{2\pi}/\sqrt{np(1-p)}}{\int_{-1}^1 e^{-\frac{x^2}{2}} dx+\sqrt{2\pi}/\sqrt{np(1-p)}}.\label{condition3}
\end{align}
Next, using \eqref{condition1} -- \eqref{condition3},
\begin{align}
&\Var(N|\mathbf{1}_{N \in I }=1) \notag\\
& \hspace{.5 in} \ge \Bigg| (np(1-p))^{\frac{1}{2}}\left( \frac{\int_{-1}^1 |x|^2 e^{-\frac{x^2}{2}}dx-2\sqrt{2\pi}/\sqrt{np(1-p)}}{\int_{-1}^1 e^{-\frac{x^2}{2}} dx+
\sqrt{2\pi}/\sqrt{np(1-p)}}\right)^{\frac{1}{2}}
\notag\\&\quad \quad \quad \quad\quad\quad\quad\quad\quad\quad\quad\quad\quad 
-\frac{2}{\int_{-1}^1
e^{-\frac{x^2}{2}}dx/{\sqrt{2 \pi}} - 1/{\sqrt{np(1-p)}}}\Bigg|^2.\label{E:lowerbound7}
\end{align}
Finally, the
estimates \eqref{E:lowerbound4-new}-\eqref{E:lowerbound7} combined with
the estimate on $\mathbb{P}(O_n)$ obtained in Theorem~\ref{theorem:On} give
the lower bound in Theorem \ref{theorem:main}, whenever $2\ln n /K^2\le h(n)
\le K_1\sqrt{n} $, where the upper bound on $h(n)$ stems from the requirement that the right hand side of \eqref{E:lowerbound4-new} needs to be lower bounded and where $K_1$ is estimated in Section~\ref{sec:constants}.






\end{proof}

\section{Proof of Theorem~\ref{theorem:On}}
\label{sec:On}
In this section, we prove the aforementioned theorem, therefore completing our proof of Theorem~\ref{theorem:main}. Before doing so, we will need to state a few definitions and set some notations used throughout the rest of the paper:

The sequences $\bm{Z}^{(k)}$ and $\bm{Y}^{(n)}$ are said to have a common subsequence of length $\ell$ if there exist increasing functions $\pi:[1,\ell] \to [1, k]$ and $\eta: [1,\ell] \to [1, n]$ such that 
 $$ Z^k_{\pi(i)} = Y_{\eta(i)}, \hspace{.5em} i = 1, 2, \ldots, \ell,$$
and $(\pi, \eta)$ is then called a pair of matching subsequences of $\bm{Z}^{(k)}$ and $\bm{Y}^{(n)}$. Also, throughout,  $M^k$ denotes the set of pairs of matching subsequences of $\bm{Z}^{(k)}$ and $\bm{Y}^{(n)}$ of maximal length. 


Following the approach in \cite{bonetto_fluctuations_2006}, the proof of Theorem \ref{theorem:On} is then divided into two cases, $k < \nu n$ and $k \ge \nu n$, where in each case $\nu < 1/m$.

\subsection{$k < \nu n$ ($\nu < 1/m$)}
\label{sec:klnun}
We begin with the simpler case $k < \nu n$. In this situation, we show that with high probability all the letters of $\bm{Z}^{(k)}$ are matched with letters of $\bm{Y}^{(n)}$. Let
$$E_{k}^{(n)} := \{ L_n(k) = k \}.$$
Then clearly, $E_{k}^{(n)}\subset E_{k-1}^{(n)} \subset \cdots \subset E_{1}^{(n)}$, and so 
$$E^{(n)}   := \bigcap_{k=1}^{\nu n} E_{k}^{(n)}=E_{\nu n}^{(n)} = \{ L_n(k+1) - L_n(k) = 1 , \forall \; k < \nu n\}. $$


\begin{lemma}
\label{lemma:ksmall}
For $\nu < 1/m$, there exists a constant $\cksmalla = \cksmalla(\nu,m)> 0$ such that,
$$\mathbb{P}(L_n(\nu n) = \nu n) \ge 1 - \exp(- \cksmalla n).$$
\end{lemma}

\begin{proof}
We construct a pair of matching sequence $(\pi, \eta)$ for $\bm{Z}^{(k)} = Z^k_1 Z^k_2 \cdots Z^k_{k}$ and $\bm{Y}$ as follows, 
$$
\begin{cases}
\pi(i) = i, \\
\eta(i) = \min \{\ell:\,\ell > \eta(i-1),Y_{\ell} = Z_i^k\},
\end{cases}
\qquad
\text{for $i\ge 1$, }
$$
where we also set $\eta(0)=0$.

Thus, $\eta(i)$ is the smallest index $\ell$ such that $Z_1^k\cdots Z_i^k$ is a subsequence of $Y_1Y_2\cdots Y_{\ell}$. In this way, $\eta(1), \eta(2), \eta(3), \cdots $ is a renewal process with geometrically distributed holding time, i.e., denoting the inter arrival times as 
$$T_i = \eta(i) - \eta(i-1),$$
then $\{T_i\}_{i \ge 1}$ is a sequence of independent geometric random variables with parameter $1/m$, i.e., 
\begin{equation*}
\mathbb{P}(T_i=t) = \frac{1}{m}\left(\frac{m-1}{m}\right)^{t-1},\; t = 1, 2, 3, \cdots.
\end{equation*}
Thus, $\mathbb{E}T_i = m$. Next,
$$\mathbb{P}(L_n(\nu n) = \nu n) 
\ge \mathbb{P}\left( \sum_{i = 1}^{\nu n}T_i < n \right)
= 1 - \mathbb{P}\left( \sum_{i = 1}^{\nu n}\left(T_i -\frac{1}{\nu}\right)\ge 0 \right),$$
and from the independence of the $\{T_i\}_{i \ge 1}$, 
\begin{align*}
\mathbb{P}\left( \sum_{i = 1}^{\nu n}\left(T_i -\frac{1}{\nu}\right)\ge 0 \right) 
& \le \inf_{s > 0} \mathbb{E} \left( e^{s\sum_{i = 1}^{\nu n}\left(T_i -\frac{1}{\nu}\right)} \right)\\
& = \inf_{s > 0} \left( \mathbb{E}  e^{s (T_0 - 1/\nu)} \right)^{\nu n} \\
& =  \inf_{s > 0}  e^{-ns} \left( \frac{e^s}{m - (m-1)e^s} \right)^{\nu n}.
\end{align*}

This last term is minimized at 
$$s = \ln  \frac{m(1 - \nu)}{m-1}, $$

thus, 
\begin{equation*}
\inf_{ s > 0}   e^{-s} \left( \frac{e^s}{m - (m-1)e^s} \right)^{\nu }  
 =  \frac{(1 - \nu)^{\nu -1}}{m(m-1)^{\nu -1 } \nu^{\nu}},
\end{equation*}
which is increasing in $\nu$ for $\nu\in(0, 1- 1/m)$. Thus,
$$\frac{(1 - \nu)^{\nu -1}}{m(m-1)^{\nu -1 } \nu^{\nu}} 
\begin{cases}
 > 1 & \text{when $\nu \in (1/m, 1 - 1/m) $} \\
 = 1 & \text{when $\nu = 1/m$} \\
 < 1 & \text{when $\nu < 1/m$.}
\end{cases}$$

Since $\nu<1/m$, by taking $\cksmalla=\ln\left({m(m-1)^{\nu -1 } \nu^{\nu}}/{(1 - \nu)^{\nu -1}}\right)$, we have
$$\mathbb{P}(E_{\nu n}^{(n)})=\mathbb{P}(L_n(\nu n) = \nu n)
\ge 1 - \exp(- \cksmalla n).$$

\end{proof}

Therefore, Lemma~\ref{lemma:ksmall} asserts that
$$\mathbb{P}(E^{(n)})=\mathbb{P}(E_{\nu n}^{(n)}) \ge 1 - \exp(- \cksmalla n).$$

\subsection{$k \ge \nu n$ ($\nu < 1/m$)}
\label{sec:kgnun}
To continue, we introduce some more definitions and notations of use throughout the section. 

\begin{enumerate}[(i)]
\item Let $\le$ denote the partial order between two increasing functions $\pi_1, \pi_2: [1, \ell] \to \mathbb{N}$, i.e., $\pi_1 \le \pi_2$ if for every $i \in [1, \ell]$, $\pi_1(i) \le \pi_2(i)$. Further $(\pi_1, \eta_1) \le (\pi_2, \eta_2)$ is short for $\pi_1 \le \pi_2$ and $\eta_1 \le \eta_2$. 

\item Let $M_{min}^k \subset M^k$ be the set of $(\pi, \eta) \in M^k$ which are minimal for the relation $\le$, i.e., such that for  $(\pi_1,\eta_1)\in M_{min}^k$ and $(\pi_2,\eta_2)\in M^k$, if $(\pi_1,\eta_1)\ge(\pi_2,\eta_2)$ then $(\pi_1,\eta_1)=(\pi_2,\eta_2)$.


\item If $(\pi, \eta)$ is a pair of matching subsequences of $\bm{Z}^{(k)}$ and $\bm{Y}^{(n)}$ of length $\ell$, a match of $(\pi, \eta)$ is then defined to be the quadruple $$\left( \pi(i), \pi(i+1), \eta(i), \eta(i+1) \right).$$
Moreover, if $\eta(i) + 2 \le \eta(i+1)$, the match is said to be non-empty. Therefore, for a non-empty match, there exists $j$, such that $\eta(i) < j < \eta(i+1)$ and $Y_j = \alpha$ for some $\alpha\in\mathcal{A}\setminus\{\alpha_{m+1}\}$. In that case, the match is said to contain an $\alpha$, and $Y_j$ is called an unmatched letter of the match $\left( \pi(i), \pi(i+1), \eta(i), \eta(i+1) \right)$. 

\item The sequence $\bm{Y}^{(n)}$ can be uniquely divided into $d$ compartments $[j_1,j_2-1], [j_2,j_3-1],\ldots,[j_d,n]$, where $1=j_1<j_2<\cdots <j_d\le n$ are determined by the following recursive relations:
\begin{equation*}
\begin{cases}
j_1=1\\
j_i = \min(n+1, \{s\in[j_{i-1} +1,n]:\,\text{$Y_{j_{i-1}}Y_{j_{i-1} +1}\cdots Y_{s}$ contains $m$ distinct letters} \}), 
\end{cases},
\end{equation*}
and $d=\max\{i:\,j_i\le n\}$.


\end{enumerate}

To get a lower bound on the probability that the length of the longest common subsequence increases by one, we recall the construction of $\bm{Z}^{(k)}$ and note that there are $(k-1)$ possible positions for the letter $U_{k+1}$ to be inserted. Therefore, $U_{k+1}$ falls into a non-empty match with probability at least ${(\text{number of nonempty matches of } (\pi, \eta))}/{(k-1)}\ge{(\text{number of nonempty matches of } (\pi, \eta))}/{k}$.   For each non-empty match, there is at least one unmatched letter, and the probability that $U_{k+1}$ takes the same value as the unmatched letter is $1/m$, resulting in the following lower bound for $(\pi, \eta) \in M^k$: 
\begin{equation}
\label{eqn:lowerbd}
\mathbb{P}\left( L_n(k+1) - L_n(k) = 1 | \bm{Z}^{(k)}, \bm{Y}^{(n)} \right) \ge \frac{1}{m} \frac{\text{number of nonempty matches of } (\pi, \eta)}{k }.
\end{equation}

Therefore, a good estimate on the number of nonempty matches of $(\pi, \eta)$ will provide a lower bound on the probability that $LC_n$ increases by one.

Next we give the main ideas behind the proof that, with high probability, the map $k \to L(k)$ is linearly increasing on $[\nu n, n]$. We use the letter-insertion scheme, described above, to prove that the random map $k\to L(k)$ typically has positive drift $\lambda$ (which will be determined later in Lemma~\ref{lemma:containing}). To do so, let 
\begin{equation}
F_{k}^{(n)} := \{ (\pi, \eta) \in M_{min}^k \text{ such that the number of nonempty matches of } (\pi, \eta) \text{ is at least }  \lambda n\}, \label{eq:fk}
\end{equation}
and let
$$F^{(n)} := \bigcap_{k = \nu n}^n F_{k}^{(n)}.$$

When $F_{k}^{(n)}$ holds, every pair of $(\pi, \eta) \in M_{min}^k$ has at least $\lambda n$ nonempty matches. Hence the number of non-empty matches divided by $k$ is larger than or equal to $\lambda n/k$. It follows from \eqref{eqn:lowerbd} that when $F_{k}^{(n)}$ holds, 
\begin{equation}
\mathbb{P}(L_n(k+1) - L_n(k) = 1 | \bm{Z}^{(k)}, \bm{Y}^{(n)}) \ge \frac{1}{m} \frac{\lambda n}{k} 
\ge \frac{\lambda}{m}>0.
\label{lb}
\end{equation}

The inequality \eqref{lb} implies that when $F^{(n)}$ holds, the map $k \to L_n(k)$ has drift at least $\lambda /m$ for $k \in [ \nu n, n]$. In other words, whenever $F^{(n)}$ holds, with high probability $k \to L_n(k)$ has positive slope on $[\nu n, n]$. 

It remains to show that, by concentration, $F^{(n)}$ holds with high probability, and this is proved by contradiction. Indeed if all the matches of $(\pi, \eta ) \in M^k $ were empty, then the following two conditions would hold:
\begin{enumerate}[(1)]
\item $(\eta(1), \eta(2), \eta(3), \cdots, \eta(\ell
) ) = (\eta(1), \eta(1)+1, \eta(1)+2, \cdots, \eta(1) + \ell-1  )$ where $\ell$ is the length of the LCS of $\bm{Z}^{(k)}$ and $\bm{Y}^{(n)}$, i.e., $\ell = L_n(k)$. 
\item The sequence 
$$Y_{\eta(1)} Y_{\eta(2)} \cdots Y_{\eta(\ell)} = Y_{\eta(1)} Y_{\eta(1)+1} \cdots Y_{\eta(1)+\ell -1}$$
would be a subsequence of 
$$Z_{\pi(1)}^k Z_{\pi(1) +1}^k \cdots Z_{\pi(\ell)}^k.$$
\end{enumerate}

Above, we have two independent sequences of i.i.d.\ uniform random variables with parameter $1/m$, where one is contained in the other as a subsequence. Thus, the longer one must approximately be at least $m$ times as long as the shorter one, hence $k$ is approximately at least $m$ times as long as $\ell = L_n(k)$. As a result, the ratio $L_n(k)/k$ is to be at most $1/m$, which is very unlikely (Lemma~\ref{lemma:Lna_lower}), leading to contradiction. \label{pg:contradiction}

From the previous arguments, it follows that with high probability any $(\pi, \eta) \in M_{min}^k$ contains a non-vanishing proportion $\epsilon >0$ of unmatched letters, hence 
$(\eta(L_n(k))-L_n(k))/\eta(L_n(k)) \ge \epsilon$, where $\eta(L_n(k))$ is the index of the last matching letter in $\bm{Y}^{(n)}$ of the match $(\pi,\eta)$. We then show that this proportion $\epsilon$ of unmatched letters generates sufficiently many non-empty matches, i.e., that the unmatched letters should not be concentrated on a too small number of matches. 

To prove that there are more than $\lambda n$ nonempty matches, the following two arguments are used:
\begin{enumerate}[(1)]
\item Any $(\pi, \eta)\in M_{min}^k$ is such that every match of $(\pi, \eta)$ contains unmatched letters from at most one compartment of $\bm{Y}^{(n)}$.  
\item There exists a $D>0$, not depending on $n$, such that, with high probability, the total number of integer points contained in the compartments of $\bm{Y}^{(n)}$ of length larger than $D$, is small. 
\end{enumerate}

Henceforth, for $(\pi, \eta) \in M_{min}^k$ the majority of unmatched letters are at most $D$ per match, ensuring that a proportion $\epsilon$ of unmatched letters implies a proportion of at least $\epsilon/D$ non-empty matches. 

\vspace{\baselineskip}

Let us return to the proof, and let $L_{\ell}(k)$ denote the length of the LCS of $\bm{Z}^{(k)}$ and $\bm{Y}^{(\ell)} = Y_1 \cdots Y_{\ell}$. In order for $\bm{Y}^{(\ell)}$ to be contained in $\bm{Z}^{(k)}$, $k$ needs to be approximately $m$ times as long as $\ell$, and, then, $L_{\ell}(k) = \ell$. Therefore, if $k = m \ell(1-\delta)$, for some $\delta = \delta(\epsilon) > 0$ not depending on $\ell$, then it is extremely unlikely that $\bm{Y}^{(\ell)}$ is a subsequence of $\bm{Z}^{(k)}$, as shown in the forthcoming lemma.

\begin{lemma}
\label{lem:plmldl}
For any $0<\delta<(m -1)/m$ and $\ell \ge 1$, we have
\begin{equation}
\mathbb{P}(L_{\ell}(m\ell(1- \delta)) = \ell) \le  e^{-\cLlaa \delta^2\ell}, 
\label{eq:lem52}
\end{equation}
where $C_2 = m/2(m-1)$. 
\end{lemma}
\begin{proof}
The proof is similar to the proof of Lemma~\ref{lemma:ksmall} and some of its notation is used.

First let $\tilde{\bm{X}}:=\tilde{\bm{X}}^{(\infty)}$, be the (infinite) subword of $\bm{X}$ with $\alpha_{m+1}$ removed, and therefore each $\tilde{\bm{X}}^{(n)}$ is a subword of $\tilde{\bm{X}}$. Next, construct a pair of matching sequence $(\pi,\eta)$ for $\tilde{\bm{X}}$ and $\bm{Y}^{(\ell)}$ as follows:
\begin{equation*}
\pi(0)=0,\quad\quad\text{ and for $i\ge 1, $}\quad
\begin{cases}
\pi(i)=\min\{j:\,j>\pi(i-1),\tilde{X}_j=Y_i\}\\
\eta(i)=i.
 \end{cases}
\end{equation*}
Thus, $\pi(i)$ is the smallest index $j$ such that  $Y_1Y_2\cdots Y_{i}$ is a subsequence of $\tilde{X}_1\cdots \tilde{X}_j$. In this way, $\pi(1), \pi(2), \pi(3), \cdots $ is a renewal process with geometrically distributed holding time, i.e., denoting the interarrival times as 
$$T_i = \pi(i) - \pi(i-1),$$
then $\{T_i\}_{i \ge 1}$ is a sequence of independent geometric random variables with parameter $1/m$, i.e., 
\begin{equation*}
\mathbb{P}(T_i=t) = \frac{1}{m}\left(\frac{m-1}{m}\right)^{t-1},\, t = 1, 2, 3, \cdots.
\end{equation*}
Thus, $\mathbb{E}T_i = m$. Then by Lemma~\ref{lemma:LCn_dist} and for $0 < \delta < 1$, we have
\begin{align*}
\mathbb{P}(L_{\ell}(m\ell(1- \delta)) = \ell)
& =
\mathbb{P}\left(\sum_{i=1}^\ell T_i\le m\ell(1-\delta)\right)
=
\mathbb{P}\left(\sum_{i=1}^\ell (m(1-\delta)-T_i)\ge 0 \right) \\
& \le 
\left( \inf_{s>0} \mathbb{E}e^{s(m(1-\delta)-T_1)}\right)^\ell
=
\left(\inf_{s>0}\frac{e^{sm(1-\delta)}}{me^s-(m-1)}\right)^\ell.
\end{align*}
This last term is minimized at 
\begin{equation*}
s = \ln\left(1+\frac{\delta}{m(1-\delta)-1}\right),
\end{equation*}
thus setting,
\begin{equation*}
w: =
\inf_{s>0}\frac{e^{sm(1-\delta)}}{me^s-(m-1)}
=\frac{(m(1-\delta)-1)\left(1+\frac{\delta}{m(1-\delta)-1}\right)^{m(1-\delta)}}{m-1}, 
\end{equation*}
it follows that, 
\begin{equation*}
\mathbb{P}\left(\sum_{i=1}^\ell (m(1-\delta)-T_i)\ge 0 \right)
\le 
e^{(\ln w)\ell}.
\end{equation*}
Now, the Taylor expansion of $\ln w$ with Lagrange remainder gives
\begin{equation*}
\ln w = -\frac{m}{2(m-1)}\delta^2+\frac{1}{6}\left(\frac{m}{(1-\xi)^2}-\frac{m^3}{(m(1-\xi)-1)^2}\right)\delta^3<-\frac{m}{2(m-1)}\delta^2,
\end{equation*}
where $0<\xi<\delta$. Letting $\cLlaa=m/2(m-1)$ finishes the proof.
\end{proof}
Lemma~\ref{lem:plmldl} further entails, as shown next, that for any $0<\epsilon<1$ there exists $\delta(\epsilon)>0$, small, such that $L_{\ell}(m\ell(1- \delta(\epsilon) ) ) \ge \ell(1-\epsilon)$ is also very unlikely.
\begin{lemma}
\label{lemma:Lla}
For any $0<\epsilon<1$ and all $\ell \ge 1$, there exists $\delta(\epsilon)>0$, with $\displaystyle\lim_{\epsilon \to 0} \delta(\epsilon) \to 0$, such that
$$\mathbb{P}\bigl((\Ethreenl)^c\bigr) \le  e^{- \cLlab \ell}, $$
where $\Ethreenl=\{L_{\ell}(m\ell(1- \delta(\epsilon) ) ) < \ell(1-\epsilon)\}$, and where $\cLlab:=(\delta(\epsilon)-\epsilon)^2\cLlaa/2$. Therefore, letting
$$\Ethreen := \bigcap_{\ell = \nu n}^n \Ethreenl,$$
it follows that,
\begin{equation}
\mathbb{P}(\Ethreen) \ge  1 - \sum_{k = \nu n}^n  e^{-\cLlab k} \ge 1 - \frac{1}{1 - e^{-\cLlab }}e^{ - \cLlab\nu n}=  1 - \cCorEthree e^{-\cCorEthreeb n},
\label{eq:pgn}
\end{equation}
where $\cCorEthree = {1}/{(1 - e^{-\cLlab })}$. 
\end{lemma}

\begin{proof}

Let $S\subset \{1, 2, \cdots, \ell \}$ have cardinality $(1-\epsilon)\ell$. Clearly, there are $\binom{\ell}{\ell(1-\epsilon)}$ such subsets $S$. Now fixing the values of $\bm{Y}^{(n)}$ at the indices belonging to $S$, there are $m^{\epsilon \ell}$ such 
  $\bm{Y}^{(n)}$ agreeing on $S$. 
Therefore,
 \begin{equation*}
\mathbb{P}\bigl((\Ethreenl)^c\bigr)\le
m^{\epsilon \ell}\binom{\ell}{\ell(1-\epsilon)}\mathbb{P}\bigl(L_{\ell(1-\epsilon)}\left(m\ell\left(1-\delta(\epsilon)\right)\bigr)=
\ell(1-\epsilon)\right).
\end{equation*}
From \eqref{eq:lem52},  
  \begin{align*}
  \mathbb{P}\bigl(L_{\ell(1-\epsilon)}\left(m\ell\left(1-\delta(\epsilon)\right)\bigr)=
\ell(1-\epsilon)\right)
&=
  \mathbb{P}\left(L_{\ell(1-\epsilon)}\left(m\left(1-\frac{\delta(\epsilon)-\epsilon}{1-\epsilon}\right)(1-\epsilon)\ell\right)=\ell(1-\epsilon)\right)\\
  &\le e^{-\cLlaa \left(\frac{\delta(\epsilon)-\epsilon}{1-\epsilon}\right)^2(1-\epsilon)\ell}
  \le
  e^{-\cLlaa (\delta(\epsilon)-\epsilon)^2\ell}.
  \end{align*}
Collecting the above estimates,
\begin{equation}
\mathbb{P}\bigl((\Ethreenl)^c\bigr)\le
m^{\epsilon \ell}\binom{\ell}{\ell(1-\epsilon)}e^{-\cLlaa (\delta(\epsilon)-\epsilon)^2\ell}.
\label{eq:llepsilon}
\end{equation}
Since
\begin{equation*}
\ell^\ell = (\epsilon\ell + (1-\epsilon)\ell)^\ell\ge\binom{\ell}{(1-\epsilon)\ell}(\epsilon\ell)^{\epsilon\ell}((1-\epsilon)\ell)^{(1-\epsilon)\ell},
\end{equation*}
then
\begin{equation*}
\binom{\ell}{(1-\epsilon)\ell}\le 
\left(
\frac{1}{\epsilon^\epsilon(1-\epsilon)^{1-\epsilon}}
\right)^\ell.
\end{equation*}
Therefore, \eqref{eq:llepsilon} becomes
\begin{equation*}
\mathbb{P}\bigl((\Ethreenl)^c\bigr)
\le e^{\left(\epsilon(\ln m-\ln\epsilon)-
(1-\epsilon)\ln (1-\epsilon)-\cLlaa (\delta(\epsilon)-\epsilon)^2\right)\ell},
\end{equation*}
%
and it is enough to choose 
\begin{equation}
\delta(\epsilon)=\epsilon+\sqrt{\frac{2}{\cLlaa}\left(\epsilon(\ln m-\ln\epsilon)-
(1-\epsilon)\ln (1-\epsilon)\right)},
\label{eq:deltaepsilon}
\end{equation} 
to obtain the stated result.
\end{proof}

Lemma~\ref{lemma:Lna_lower} and Lemma~\ref{lemma:L_choose_delta}, presented next, formalize our contradictory argument asserted above. To show that it is very unlikely that ``the ratio $L_n(k)/k$ is at most $1/m$'',
 note, at first, that for $n\ge 2$,
\begin{equation}
    \mathbb{E} L_{n}(n)
    >
    \mathbb{E} \sum_{i=1}^{n}\mathbbm{1}_{\{Y_i=Z^{n}_i\}}
    \ge
    n \mathbb{P}(Y_1=Z^{n}_1) =\frac{n}{m}.
    \label{eq:en0}
\end{equation}
Specifically, when $n=2$, see \cite{chvatal_longest_1975},
\begin{equation}
    \E{L_{2}(2)}/2=\frac{4m^2-5m+3}{2m^3}.
    \label{eq:en02}
\end{equation}
Now, choose $\xi_m$ such that 
\begin{equation}
1/m<\xi_m<\E{L_{2}(2)}/{2},
\label{eq:ELn0bound}
\end{equation}
and let us show that very likely $L_n(k)/k$ is larger than $\xi_m$. To do so, let
$$\Efournk := \{L_n(k) \ge \xi_m k \},$$
and
$$\Efourn := \bigcap_{k = \nu n}^n \Efournk.$$
\begin{lemma}
\label{lemma:Lna_lower}
There exist constants $\cLnalowera, \cLnalowerb > 0$, such that
\begin{equation}
\mathbb{P}(\Efourn) \ge 1 - \cLnalowera e^{- \cLnalowerb n}. 
\label{eq:Lna_lower}
\end{equation}
\end{lemma}

\begin{proof}
Divide the sequences $\bm{Z}^{(k)}$ and $\bm{Y}^{(n)}$ into subsequences of length
  2, as given in the previous lemma. Then, by superadditivity, $L_k(k)\ge\sum_{i=1}^{k/2}\tildel{L}_i$, where $\tildel{L}_i$ is the length of
 the  longest common subsequence
  between $Y_{2(i-1)+1}Y_{2i}$ and $Z^k_{2(i-1)+1}
  Z^k_{2i}$.  Clearly, by the i.i.d.\ assumptions, $\mathbb{E}(\tildel{L}_i)=\mathbb{E}({L}_{2}(2))$ is constant. Hence for $\tau>0$,
\begin{equation}
\mathbb{P}\left(\sum_{i=1}^{k/2}\tildel{L}_i<k\left(\frac{\mathbb{E}(\tildel{L}_i)-\tau}{2}\right)\right)\le
\left(\inf_{s<0}\mathbb{E}\left(e^{s\left(\tildel{L}_1-(\mathbb{E}({L}_{2}(2))-\tau)
\right)}\right)\right)^{\frac{k}{2}}.
\end{equation}

Now let
$p(s,\tau):=\mathbb{E}\left(e^{s\left(\tildel{L}_1-(\mathbb{E}({L}_{2}(2))-\tau)\right)}\right)$, it is
easy to see that $p(s,\tau)$ is smooth in $s$, 
and that
\begin{equation*}
\begin{cases}
p(0,\tau)=1,\\
\frac{\partial p(s,\tau)}{\partial s}|_{s = 0}=\tau>0,
\end{cases}
\end{equation*}
 for every $\tau>0$. Hence,
\begin{equation}
\inf_{s<0} p(s,\tau)<e^{-c(\tau)}, 
\end{equation}
for a suitable $c(\tau)>0$. Thus, 
$$
\mathbb{P}((\Efournk)^{c})\le \mathbb{P}(L_k(k)<\xi_{m}k) \le \mathbb{P}\left(\sum_{i=1}^{k/2}\tildel{L}_i<k\left(\frac{\mathbb{E}(\tildel{L}_i)-\tau}{2}\right)\right)<e^{-c(\tau)k/2}.
$$
Now, let $\tau=\tau_m:=\mathbb{E}({L}_{2}(2))-2\xi_{m}$, let $\xi_m=11/10m$, and so
\begin{equation*}
    p(s,\tau_m)=\frac{e^{-11s/5m}\left(m e^{2 s}+(4 m^2-7 m+3) e^s+m^3-4 m^2+6 m-3\right)}{m^3}.
\end{equation*}
Since $\inf_{s<0}p(s,\tau_m)<e^{-1/1000m}$, one can choose $c(\tau_m)=1/1000m$.
Hence,
\begin{equation*}
\mathbb{P}((\Efourn)^{c})\le \sum_{k=\nu n}^{n}e^{-c(\tau_m)k/2}
= 
\frac{e^{c(\tau_m)(1-n\nu)/2}-e^{-c(\tau_m)n/2}}{e^{c(\tau_m)/2}-1}
\le 
\frac{e^{c(\tau_m)/2}}{e^{c(\tau_m)/2}-1} e^{c(\tau_m)(-n\nu)/2}. 
\end{equation*}
Choosing $\cLnalowera = \left.{e^{c(\tau_m)/2}}\middle/{(e^{c(\tau_m)/2}-1)}\right.$, and $\cLnalowerb  = c(\tau_m)(\nu)/2 $, we have,
$$\mathbb{P}(\Efourn) \ge 1 - \cLnalowera e^{- \cLnalowerb n}. $$
\end{proof}
We now finish our argument showing that, with high probability, any $(\pi, \eta) \in M_{min}^k$ contains a non-vanishing proportion $\epsilon >0$ of unmatched letters. To do so, let $$\Esixnk := \{ L_n(k) \le (1 - \epsilon) \eta(L_n(k)), \text{ for } (\pi, \eta) \in M_{min}^k \}, $$ be the event that any pair of matching subsequences $(\pi, \eta) \in M_{min}^k$ has a proportion at least $\epsilon$ of unmatched letters, and
let
$$\Esixn := \bigcap_{k = \nu n }^n \Esixnk.$$

Above, $\eta(L_n(k)) - L_n(k)$ is the number of unmatched letters, since $\eta(L_n(k))$ is the position of the last matched letter, while $L_n(k)$ is the number of matched letters. 

\begin{lemma}
\label{lemma:L_choose_delta}
Let $\epsilon>0$ be small enough such that $\delta(\epsilon)$, as given in \eqref{eq:deltaepsilon}, satisfies
\begin{equation}
\frac{1}{1 - \delta(\epsilon)} < \xi_{m}m,
\label{numbers}
\end{equation}
where $\xi_m$ is as in \eqref{eq:ELn0bound}.
Then, for all $k \ge \nu n$, 
\begin{equation}
\Ethreen \cap \Efournk \subset \Esixnk, 
\label{3k4k6k}
\end{equation}
and thus 
\begin{equation}
\Ethreen \cap \Efourn \subset \Esixn.
\label{3n4n6n}
\end{equation}
\end{lemma}

\begin{proof}
Let $k\in[\nu n,n]$.
 In order to prove \eqref{3k4k6k}, we show that if $\Esixnk$ does not hold while $\Ethreen$ does
hold, then $\Efournk$ does not hold either. 
Let $(\pi,\eta)\in M_{min}^k$. If $\Esixnk$ does not hold, than the proportion
of unmatched letters of $(\pi,\eta)$ is smaller than $\epsilon$, i.e.,
$$\frac{L_\ell(k)}{\ell}\geq 1-\epsilon, $$
where $\ell:=\eta(L_n(k))$. (Note that $L_\ell(k)=L_n(k)$, since
$(\pi,\eta)$ is of maximal length.) Therefore,
\begin{equation}
\label{I}
L_\ell(k)\geq \ell(1-\epsilon).
\end{equation}
Now, when $\Ethreenl$ holds, then
\begin{equation}
\label{II}
L_\ell(ml(1-\delta(\epsilon))) < \ell(1-\epsilon).
\end{equation}
Comparing \eqref{I} with \eqref{II} and noting that
the (random) map $x\mapsto L_\ell(x)$ is increasing, yield

$$k\geq m\ell(1-\delta(\epsilon)), $$
and thus
$$k\geq m\eta(L_n(k))(1-\delta(\epsilon))
\geq mL_n(k)(1-\delta(\epsilon)).$$
Hence, from \eqref{numbers},
\begin{equation*}
\frac{L_n(k)}{k}\leq\frac{1}{m(1-\delta(\epsilon))}<\xi_{m},
\end{equation*}
which implies that $\Efournk$ cannot hold.
\end{proof}
As an example, when $\epsilon\le e^{-9}/(1+\ln{m})$, 
\begin{align*}
    \delta(\epsilon)&=\epsilon+\sqrt{\frac{2}{\cLlaa}\left(\epsilon(\ln m-\ln\epsilon)-
(1-\epsilon)\ln (1-\epsilon)\right)}\\
&\le \epsilon + 2\sqrt{(1+\ln m-\ln\epsilon)\epsilon}\\
&\le e^{-9} + 2\sqrt{10e^{-9}}\\
&<\frac{1}{11},
\end{align*}
and therefore,
\begin{equation*}
    \frac{1}{1-\delta(\epsilon)}<\frac{10}{11}=\xi_m m.
\end{equation*}

In order to estimate the event $F^{(n)}$, we need to show that the unmatched letters of $\bm{Y}^{(n)}$ do not concentrate in a small number of matches of $(\pi, \eta)\in M_{min}^k$. From the minimality of $M_{min}^k$, the unmatched letters of a match of $(\pi, \eta)\in M_{min}^k$ contain at most one compartment. 

Let $N^D$ be the total number of letters in the sequence $\bm{Y}^{(n)}$ contained in a compartment of length at least $D$, and let,
$$\Efiven := \{ N^D \le \xi_{m} \epsilon\nu n/2 \},$$
where again $\xi_m$ is given via \eqref{eq:ELn0bound}.

\begin{lemma}
\label{lemma:block}
For any $0<\epsilon <1$, there exist a positive integer $D$, and positive constant $\cBlocka$ and $\cBlockb$ depending on $D$, such that 
\begin{equation}
\mathbb{P}(\Efiven) \ge 1 - \cBlocka e^{-\cBlockb n}.
\label{eq:Pjn}
\end{equation}
\end{lemma}

\begin{proof}
Let
$\tilde{N}^D$ be the number of integers $s\in [0,n-D]$
such that
\begin{equation}
\label{Y=}
(Y_s,Y_{s+1},\ldots,Y_{s+D-1}) \text{ belongs to a compartment}.
\end{equation}
It is easy to check that
\begin{equation}
 N^D\leq D \tilde{N}^D.
 \label{eq:nddnd}
\end{equation}
Let now $\tilde Y_s$, $s\in [0,n-D]$, be equal to 1 if and only if \eqref{Y=} holds, and
0 otherwise. Clearly,
\begin{equation}
\label{sumtilde}
\sum_{s=1}^n\tilde{Y}_s=\tilde{N}^D.
\end{equation}
To estimate the sum \eqref{sumtilde}, decompose
it into $D$ subsums of i.i.d.\ random variables $\Sigma_1,\Sigma_2,\ldots,\Sigma_{D}$
where 
\begin{equation*}
\Sigma_i=\sum_{\genfrac{}{}{0pt}{}{s=1,\ldots,n }{ s\; {\rm mod}\;
    D=i}} \tilde Y_s, 
\end{equation*}
so that 
\begin{equation}\label{sumD}
\tilde N^D=\sum_{i=1}^D \Sigma_i. 
\end{equation}
Then, from \eqref{eq:nddnd}
\begin{equation}
\mathbb{P}\left(N^D>\frac{\xi_{m} \epsilon \nu }{2} n\right)\leq
\mathbb{P}\left(\tilde{N}^D>\frac{\xi_{m} \epsilon \nu }{2D} n\right)\leq
D \mathbb{P}\left(\Sigma_1>\frac{\xi_{m} \nu \epsilon}{2D^2} n\right), 
\end{equation}
since in \eqref{sumD} at least one of
the summands has to be larger than
$n{\xi_{m} \epsilon \nu }/2D^2$.
Now, the $\tilde{Y}_s$ appearing in the subsum $\Sigma_1$ are
i.i.d.~Bernoulli random variables with 
$$\mathbb{P}(\tilde{Y}_s=1)\le m \left(\frac{m-1}{m}\right)^D.$$ 
Therefore, 
\begin{equation}
\mathbb{P}\left(\Sigma_1>(\mathbb{E}\tilde{Y}_s+\delta)\frac{n}{D}\right)\leq e^{-c(\delta)
  \frac{n}{D}},
\end{equation}
with $c(\delta)>0$ for $\delta>0$. Take $\delta=\mathbb{P}(\tilde{Y}_s=0)=1-\mathbb{P}(\tilde{Y}_s=1)$, then $c(\delta)=-\ln{\mathbb{P}(\tilde{Y}_s=1)}$.
Thus it is enough to choose $D$ such that 
\begin{equation}
2Dm \left(({m-1})/{m}\right)^D<{\xi_{m} \nu\epsilon}.
\label{eq:chooseD}
\end{equation}
 Let $x=(m-1)/m$, $y=\xi_m\nu\epsilon/2m$, we next show that, 
\begin{equation}
    D=\frac{1}{y(\ln x)^2}=\frac{40 e^9 m^3 (1+\ln m)}{11 \ln ^2\left(\frac{m-1}{m}\right)},
    \label{eq:D}
\end{equation}
does satisfy \eqref{eq:chooseD}, or equivalently that $Dx^D<y$. With the choice in \eqref{eq:D}, $Dx^D<y$ is equivalent to $2y\ln x\ln y+2y(\ln x)^2<1$, which is true since 
\begin{equation*}
2y\ln x\ln y+2y(\ln x)^2=2(-\ln x)(-y\ln y)+2y(\ln x)^2
\le 2\ln2\cdot 9e^{-9}+2(\ln2)^2e^{-9}<1.
\end{equation*}
Choosing $\cBlocka = D$ and $\cBlockb = c(\delta)/D$, we have 
$$\mathbb{P}(\Efiven) \ge 1 - \cBlocka e^{-\cBlockb n}.$$
\end{proof}

We can now find a suitable $\lambda$ such that when $H^{(n)}$, $I^{(n)}$ and $J^{(n)}$ all hold, then $F^{(n)}$ (which depends on $\lambda$, see \eqref{eq:fk}) also holds.
\begin{lemma}
\label{lemma:containing}
Let $\epsilon>0$ be as in Lemma~\ref{lemma:L_choose_delta}, let $D$ be such that $2Dm \left(({m-1})/{m}\right)^D<{\xi_{m} \nu\epsilon}$, and let 
$$\lambda= \frac{\xi_{m} \nu}{2} \frac{\epsilon}{D-1}.$$ Then, for $k \ge \nu n$,
\begin{equation}
\Efourn \cap \Efiven \cap \Esixnk \subset F_{k}^{(n)}, 
\label{12subset3}
\end{equation}
and thus
\begin{equation}
\Efourn  \cap \Efiven \cap \Esixn \subset F^{(n)}.
\label{eq:12subset32}
\end{equation}
\end{lemma}

\begin{proof}
  We prove \eqref{12subset3}, from which \eqref{eq:12subset32} immediately follows. On $\Esixnk$, each $(\pi,\eta)\in
  M_{min}^k$ has at least $\epsilon \eta(L_n(k))$ unmatched letters.
 But,
\begin{equation}
\label{triv1}
 \eta(L_n(k))\geq L_n(k).
\end{equation}
  When $\Efourn$
holds,
\begin{equation}
\label{triv2}L_n(k)\geq \xi_m k.
\end{equation}
Since $k\geq \nu n$, \eqref{triv1} and
\eqref{triv2}, together imply
that  the number of unmatched letters of $(\pi,\eta)\in M_{min}^k$ is at least    
$\epsilon\; \xi_{m} \nu n.$  By $\Efiven$,
 there are at most  $\xi_{m} \nu\epsilon n/2$ letters contained in 
compartments of length at least $D$.  Thus, there
are at least $\xi_{m} \nu\epsilon n/2$ unmatched letters contained in compartments
of length less than $D$.  But, every match of  $(\pi,\eta)\in M_{min}^k$ contains
unmatched letters from only one compartment, and as such every
match can contain at most
$D-1$ 
unmatched letters from compartments of length less than $D$.  Therefore, these $\epsilon\,\xi_{m} \nu n/2$ unmatched letters
which are not in $N^D$, must fill at least $\epsilon\,\xi_{m} \nu n/(2D-2)$
 matches of  $(\pi,\eta)\in M_{min}^k$. Hence,
$(\pi,\eta)\in M_{min}^k$ has at least
$\epsilon\,\xi_{m} \nu n/(2D-2)$ non-empty matches.
\end{proof}

\noindent Combining Lemma~\ref{lemma:L_choose_delta} and Lemma~\ref{lemma:containing} gives,
$$\mathbb{P}((F^{(n)})^c) \le \mathbb{P}((\Ethreen)^{c}) + \mathbb{P}((\Efourn)^{c}) + \mathbb{P}((\Efiven)^{c}),$$
which via \eqref{eq:pgn}, \eqref{eq:Lna_lower}, and \eqref{eq:Pjn} entails 
$$ \mathbb{P}(F^{(n)}) \ge 1 -  \cCorEthree e^{-\cCorEthreeb n}  - \cLnalowera e^{- \cLnalowerb n}  - \cBlocka e^{-\cBlockb n}. $$
Next, recalling the definition of $O_n$ in \eqref{eq:On}, observe that 
$$\mathbb{P}(O_n^c) \le \mathbb{P}(O_n^c \cap E^{(n)} \cap F^{(n)}) + \mathbb{P}((F^{(n)})^c) + \mathbb{P}((E^{(n)})^c).$$

The next result estimates the first probability, on the above right hand side, and, therefore, completes the proof of Theorem~\ref{theorem:On}. 

\begin{lemma}
\label{lemma:On_intersection}
Let $K\le 1/2m$, then
$$\mathbb{P}(O_n^c \cap E^{(n)} \cap F^{(n)}) \le ne^{-2K^2h(n)}.$$
\end{lemma}

\begin{proof}
Let $\lambda$ given as in Lemma~\ref{lemma:containing} be at most 1,  and let
$K:= \lambda/{2m}$, so that $K\leq 1/{2m}$.
Let 
$$\Delta(k):= \begin{cases} 
 L_n(k+1)-L_n(k) & \text{when $F^{(n)}_{k}$
holds,  } \\
 1 & \text{ otherwise. }
\end{cases}$$
From 
\eqref{lb}, it follows that:
\begin{equation}
\label{martingale}
\mathbb{P}\left .\left(\Delta(k)=1\right|\sigma_k\right)\geq \lambda/m ,
\end{equation}
where $\sigma_k$ denote the $\sigma$-field generated by the $Z_i^k$ and $Y_j$, namely,
\begin{equation*}
\sigma(Z_i^k,Y_j\mid i\le k,j\le n).
\end{equation*}
Moreover, $\Delta(k)$ is equal to zero or one (since $L_n(\cdot)$ is non-decreasing on $\mathbb{N}$)
and is also $\sigma_k$-measurable.
Let 
$$\tilde{L}_n(k) = 
\begin{cases}
L_n( \nu n)+
\sum_{i=\nu n}^{k-1}\Delta(i) & \text{ for } k\in [\nu n,n], \\
L_n(k) & \text{ for } k\in [0,\nu n].
\end{cases}$$
Note that when $F^{(n)}$
holds, then
\begin{equation}
\label{Idontkno}
L(k)=\tilde{L}(k),
\end{equation} 
for all $k\in [0,n-1]$. 
Define
$$\tilde{O}^{(n)}_{i,j}=\{\tilde{L}_n(j)-\tilde{L}_n(i)\geq K(j-i)\}, $$
and
$$\tilde{O}_{n}=\bigcap_{\substack{i,j \in I\cap[\nu n, n]\\j  \ge i + h(n)}}\tilde{O}^{(n)}_{i,j}. $$

When $E^{(n)}$ holds, then $L_n(k)$ has a slope of 
one on the domain $[0,\nu n]$. Therefore, since $K\leq {1}/{2m}$, the slope
condition of $O_{n}$ holds on the
domain $[0,\nu n]\cap I$.
When $F^{(n)}$ holds, then $L_n(k)$ and $\tilde{L}_n(k)$
are equal.  Therefore, when $F^{(n)}$ and
$\tilde{O}_n$ both hold, then the slope condition
of $O_n$ is verified on the domain $[\nu n,n]\cap I$.  Hence,
\begin{equation}
\label{sex}
E^{(n)}\cap F^{(n)} \cap \tilde{O}_n= E^{(n)}\cap F^{(n)}\cap O_n,
\end{equation}
and thus
$$\mathbb{P}(O_{n}^{c}\cap E^{(n)}\cap F^{(n)})
=\mathbb{P}(\tilde{O}_{n}^{c}\cap E^{(n)}\cap F^{(n)})
\leq \mathbb{P}(\tilde{O}_{n}^{c} ).$$
It only remains to estimate $\mathbb{P}(\tilde{O}_{n}^{c})$.  First,
\begin{equation}
\label{final}
\mathbb{P}(\tilde{O}_n^{c})\leq 
\sum_{\substack{i,j \in I\cap[\nu n, n]\\j  \ge i + h(n)}}\mathbb{P}((\tilde{O}^{(n)}_{i,j})^c). 
\end{equation}
Then, from Hoeffding's exponential inequality, for any $t>0$,
\begin{equation}
\mathbb{P}\left(\frac{\sum_{s=i}^j\Delta(s)}{j-i}<\mathbb{E}\Delta(i)-t\right)
<
e^{-2(j-i)t^2}.
\label{eq:apply-hoeffding}
\end{equation}
With the help of \eqref{martingale}, and since $K=\lambda/2m$, by choosing $t=\mathbb{E}\Delta(i)-K$, \eqref{eq:apply-hoeffding} becomes
\begin{equation}
\label{martingale2}
\mathbb{P}((\tilde{O}^{(n)}_{i,j})^c)\leq e^{-2|i-j|(\mathbb{E}\Delta(i)-K)^2}
\le e^{-2K^2h(n)}, 
\end{equation}
for all $i,j\in[\nu n,n]$.
Then, note that
there are at most $n$ terms in the sum in
\eqref{final}. Thus
 \eqref{final} and \eqref{martingale2} together imply that
\begin{equation}
\label{martingale3}
\mathbb{P}(\tilde{O}_n^{c})\leq n e^{-2K^2h(n)}.
\end{equation}

\end{proof}

\section{Estimation of the Constants}
\label{sec:constants}
To estimate $C$ in \eqref{eq:main}, we need to first  estimate various constants.

First let $\nu=1/2m$. Next, to estimate $K_1$, the right hand side of \eqref{E:lowerbound4-new} needs to be lower bounded. When $n\ge 900/(p(1-p))$, \eqref{E:lowerbound7} gives that 
\begin{equation*}
\Var(N|\mathbf{1}_{N \in I }=1)\ge\frac{1}{1000}p(1-p)n.
\end{equation*}
Therefore, any $K_1$ satisfying $0<K_1<\sqrt{p(1-p)}/(10\sqrt{10})$ is fine. Choosing $K_1=\sqrt{p(1-p)}/(20\sqrt{5})$, then
\begin{equation*}
    \Var(N|\mathbf{1}_{N \in I }=1)-h(n)^2\ge \frac{1}{2000}p(1-p)n.
\end{equation*}

To estimate $A$ and $B$ in \eqref{eq:thmOn} requires upper bounds on $\cCorEthree$, $\cLnalowera$, $\cBlocka$ and lower bounds for $\cLlab$,  $\cLnalowerb$, $\cBlockb$. As shown after Lemma~\ref{lemma:L_choose_delta}, we can choose $\epsilon=e^{-9}/(1+\ln{m})$, then
\begin{align*}
      \cLlab &= (\delta(\epsilon)-\epsilon)^2\cLlaa/2\\
    &= \epsilon\ln m -(1-\epsilon)\ln{(1-\epsilon)}-\epsilon\ln{\epsilon}\\
    &\ge\epsilon\ln m\ge e^{-10},
\end{align*}
and 
\begin{equation*}
\cCorEthree ={1}/{(1 - e^{-\cLlab })}\le e^{11}.
\end{equation*}
Lemma~\ref{lemma:Lna_lower} gives 
\begin{equation*}
    \cLnalowera=\frac{e^{1/2000m}}{e^{1/2000m}-1}\le 1+2000m,
\end{equation*}
and
\begin{equation*}
    \cLnalowerb = \frac{1}{4000m^2}.
\end{equation*}
Lemma~\ref{lemma:block} gives 
\begin{equation*}
    \cBlocka = D \le 20e^9,
\end{equation*}
and
\begin{align*}
    \cBlockb &= \frac{c(\delta)}{D}\\
    &=\frac{-\ln\mathbb{P}(\tilde{Y}_s=1)}{D}\\
    &\ge \ln\frac{m}{m-1}-\frac{\ln m}{D}\\
    &=\ln\frac{m}{m-1}\left(
    1-\frac{11\ln m\ln\frac{m}{m-1}}{40e^9m^3(1+\ln m)}
    \right)\\
    &\ge \ln\frac{m}{m-1}\left(1-\frac{1}{20e^9}\right)
    \ge \frac{1}{2m}.
\end{align*}
Therefore, one can take $A=\max\{1+2000m,20e^9\}$ and $B=e^{-10}/m^2$. Then, for $n\ge e^{10}m^2\ln{(80e^9+8000m)}$, $\mathbb{P}(O_n)\ge 1/2$

Note that when $n\ge 400/(p(1-p))$, we also have $\mathbb{P}(N\in I)\ge 1/2$. Let 
\begin{equation*}
    \cfinala=\frac{K^2}{64000}p(1-p),
\end{equation*}
and let
\begin{equation*}
    \cfinalb=\min_{n\le \max\{900/(p(1-p)),e^{10}m^2\ln{(80e^9+8000m)}\}}
    \frac{\Var LC_n}{n},
\end{equation*}
then one can choose $C=\min\{\cfinala,\cfinalb\}$ in \eqref{eq:main}. 
\section{Concluding Remarks}
\begin{itemize}
    \item The results of the paper show that we can approach as closely as we want the uniform case and have a linear order on the variance of $LC_n$. However, the lower order of the variance in the uniform case is still unknown although numerical results, see \cite{liu_2017}, leave little doubt that the variance is linear in the length of the words. (Unfortunately, the estimates of the previous section, on $C=C(p,m)$ in \eqref{eq:main}, converge to zero as $p\to 0$.)   
    \item Combining the above results with techniques and results presented in \cite{houdre_order_2016}, the upper and lower bound obtained above can be generalized to provide estimates of order $n^{r/2}$, $r\ge 1$, on the centered $r$-th moment of $LC_n$.

   \item  Finally, the above results might also be extended to the general case where the letters of one sequence are taken with probability $p_i$, $i = 1, 2,\ldots,m$, 
where $p_i>0$ and $\sum_{i=1}^m p_i = 1$, while for the other sequence the first $m$ letters are taken with probability $p_i -  r_i>0$ and the 
extra letter is taken with probability $p=\sum_{i=1}^m r_i$.  
Then many of the lemmas remain true replacing $1/m$ by $\inf_{i=1,\ldots, m} p_i$  or $\inf_{i=1,\ldots, m} (p_i - r_i)/(1-\sum_{k=1}^m r_k)$. For example, in the heading of Section~\ref{sec:klnun} and Section~\ref{sec:kgnun}, in \eqref{eqn:lowerbd}, \eqref{lb}, \eqref{eq:en0}, and Lemma~\ref{lemma:On_intersection}, the $1/m$ can be replaced by $\inf_{i=1,\ldots, m} p_i$. 
In \eqref{eq:lem52} of Lemma~\ref{lem:plmldl}, and in the definition of $\Ethreenl$ in Lemma~\ref{lemma:Lla}, the term $L_\ell(m\ell(1-\delta))$ would have to be replaced with
\begin{equation*}
  L_\ell\left(\frac{\ell(1-\delta)(1-\sum_{k=1}^m r_k)}{\inf_{i=1,\ldots, m} (p_i - r_i)}\right). 
\end{equation*}
However, some constants that needs delicate estimations, such as $\xi_m$, could be a further research topic.

\end{itemize}

\bibliographystyle{plain}
\bibliography{lcs}

\end{document}